\documentclass[10pt,reqno]{amsart}

\usepackage{verbatim}
\usepackage{amssymb}
\usepackage{enumerate, xspace}
\usepackage{marvosym} % monetary symbols, \EUR or \EURcr, \EyesDollar
\usepackage[sans]{dsfont} % allows \mathds{E 1}, without sans is less heavy 
\usepackage{bbm}
\usepackage{changebar}
\usepackage[colorlinks]{hyperref}
\usepackage[usenames,dvipsnames,svgnames,table]{xcolor}
\usepackage[charter]{mathdesign}
\usepackage{dsfont}
\usepackage{ulem}
\usepackage{setspace}

\numberwithin{equation}{section}
\newtheorem{Theorem}{Theorem}
\newtheorem{Proposition}{Proposition}

\newtheorem{Lemma}{Lemma}

\newtheorem{Remark}{Remark}

\newcommand{\R}{\mathbb{R}}

\newcommand{\Z}{\mathbb{Z}}
\newcommand{\N}{\mathbb{N}}

\newcommand{\dd}{\mathrm{d}}

\newcommand{\ed}{\mathrm{e}}

\newcommand{\supp}{\mathrm{supp}}
\newcommand{\Const}{\mathrm{C}}
\newcommand{\const}{c}

\newcommand{\lgs}[1]
{\langle\hspace{-0.148cm}\langle\hspace{-0.148cm} \langle\;  #1 \; \rangle\hspace{-0.148cm}\rangle\hspace{-0.148cm} \rangle}
\newcommand{\blgs}[1]
{\big\langle\hspace{-0.165cm}\big\langle\hspace{-0.165cm}\big\langle\;  #1\; \big\rangle\hspace{-0.165cm}\big\rangle\hspace{-0.165cm}\big\rangle}
\newcommand{\Blgs}[1]
{\Big\langle\hspace{-0.192cm}\Big\langle\hspace{-0.192cm}\Big\langle \; #1\; \Big\rangle\hspace{-0.192cm}\Big\rangle\hspace{-0.192cm}\Big\rangle}
\newcommand{\BBlgs}[1]
{\bigg\langle\hspace{-0.24cm}\bigg\langle\hspace{-0.24cm}\bigg\langle \; #1\; \bigg\rangle\hspace{-0.24cm}\bigg\rangle\hspace{-0.24cm}\bigg\rangle}

\newcommand{\ZZ}{\mathbb{Z}}                                             
                                              
\newcommand{\RR}{\mathbb{R}}

\renewcommand{\le}{\leq}
\renewcommand{\ge}{\geq}

\newcommand{\benum}{\begin{enumerate}}
\newcommand{\eenum}{\end{enumerate}}
\newcommand{\bitem}{\begin{itemize}}
\newcommand{\eitem}{\end{itemize}}

\newcommand{\barray}{\begin{array}}
\newcommand{\earray}{\end{array}}

\newcommand{\vertiii}[1]{{\left\vert\kern-0.25ex\left\vert\kern-0.25ex\left\vert #1 
 \right\vert\kern-0.25ex\right\vert\kern-0.25ex\right\vert}}

\begin{document}

% title
\title{Hydrodynamic limit for a disordered harmonic chain}

\author{C\'edric Bernardin}
\address{C\'edric Bernardin\\
Laboratoire J.A. Dieudonn\'e
UMR CNRS 7351\\
Universit\'e de Nice Sophia-Antipolis,
Parc Valrose\\
06108 NICE Cedex 02, France}
\email{{\tt cbernard@unice.fr}}

\author{Fran\c{c}ois Huveneers}
\address{Fran\c{c}ois Huveneers\\
CEREMADE\\
UMR-CNRS 7534\\
Universit\'e de Paris Dauphine, PSL Research University\\
Place du Mar\'echal De Lattre De Tassigny\\
75775 Paris Cedex 16, France}
\email{{\tt huveneers@ceremade.dauphine.fr}}

\author{Stefano Olla}
\address{Stefano Olla\\
CEREMADE\\
UMR-CNRS 7534\\
Universit\'e de Paris Dauphine, PSL Research University\\
Place du Mar\'echal De Lattre De Tassigny\\
75775 Paris Cedex 16, France}
\email{{\tt olla@ceremade.dauphine.fr}}

\date{\today}

\thanks{Work partially supported by the grants ANR-15-CE40-0020-01 LSD and ANR-14-CE25-0011 EDNHS of the French National Research Agency (ANR)}

\maketitle

\begin{abstract}
We consider a one-dimensional unpinned chain of harmonic oscillators with random masses.
We prove that after hyperbolic scaling of space and time the distributions of the elongation, momentum and energy
 converge to the solution of the Euler equations.  
Anderson localization decouples the mechanical modes from the thermal modes, allowing the closure of the energy conservation equation even out of thermal equilibrium.
This example shows that the derivation of Euler equations rests primarily on scales separation and not on ergodicity. 
{Furthermore it follows from our proof that the temperature profile does not evolve in any space-time scale.} 
%that provides a dynamical proof of the vanishing thermal conductivity of the disordered harmonic chain
\end{abstract}

\section{Introduction}\label{sec:introduction}
In this paper, we consider a one-dimensional chain of coupled harmonic oscillators with masses $m_x$ and Hamiltonian 
$$
	H \; = \; \sum_{x\in \Z} \left( \frac{p_x^2}{2 m_x} + g\frac{(q_{x+1} - q_{x})^2}{2} \right) .% \; = \; \sum_{x\in \Z} \left( \frac{p_x^2}{2 m_x} + g\frac{r_x^2}{2} \right)
$$
By changing units, one can assume that the stiffness coefficient $g$ is equal to $1$.
The dynamics is governed by Hamilton's equations: 
$$
m_x \dot q_x =p_x, \qquad \dot p_x = (\Delta q)_x,
$$ 
where we have used the notation $\Delta = \nabla_- \nabla_+ = \nabla_+ \nabla_-$ for the discrete Laplacian, 
with $(\nabla_+ f)_x = f_{x+1} - f_{x}$ and $(\nabla_- f)_x = f_x - f_{x-1}$. 
{For the sake of simplicity, we consider for now the system on the infinite lattice $\Z$. 
This is by no means necessary, and starting from the next section, we will restrict ourselves to a finite box and take the thermodynamic limit properly.}

This system was first analyzed in finite volume when all masses $m_x$ are equal. 
Putting the chain in a non-equilibrium stationary state (NESS) between two heat reservoirs at different temperatures, 
it was found in \cite{rieder} that the energy current does not decay with the size of the system, indicating that energy propagates ballistically.
The situation changes if the masses are taken to be i.i.d.\@ random variables. 
This case was first investigated in \cite{rubin_greer,casher_lebowitz} and subsequently studied in \cite{verheggen,dhar,ajanki_huveneers}. 
As it turns out, the disordered harmonic chain is an Anderson insulator in disguise \cite{anderson}. 
However, as a consequence of the conservation of momentum, 
the ground state of the operator $M^{-1}\Delta$, featuring in Newton's equation $\ddot q = M^{-1}\Delta q$ with $M$ the diagonal matrix of the masses,
is a ``symmetry protected mode'' \cite{halperin}, implying a divergent localization length in the lower edge of the spectrum. 
This leads to a rich and unexpected phenomenology. 
In particular, if the chain is again in a NESS, the scaling of the energy current with the system size 
happens to depend on boundary conditions and spectral factors of the reservoirs \cite{dhar}.
This finding reveals also the complete lack of local thermal equilibrium, 
that results eventually from integrability (see Section \ref{subsec: invariant quantities}).

The harmonic chain has three ``obvious'' conserved quantities: 
the total energy $H$, the total momentum $P = \sum_x p_x$ and the total stretch or elongation $R = \sum_x r_x$ with $r_x = (\nabla_+ q)_x$.   
This gives rise to the following microscopic conservation laws: 
$$
	\dot r_x = \frac{p_{x+1}}{m_{x+1}} - \frac{p_x}{m_x}, 
	\qquad
	\dot p_x = r_x - r_{x-1}, 
	\qquad 
	\dot e_x = \frac{r_x p_{x+1}}{m_{x+1}} - \frac{r_{x-1}p_x}{m_x}
$$
with $e_x = \frac{1}{2} \left( \frac{p_x^2 }{ m_x } + r_x^2 \right)$. 
After a hyperbolic rescaling of space and time, we ask in this paper whether the empirical densities of these conserved quantities converge to the densities $\mathbf r, \mathbf p$ and $\mathbf e$ governed by the macroscopic laws
$$
	\partial_t \mathbf r (y,t) = \frac{1}{\overline{m}} \partial_y \mathbf p (y,t), 
	\quad 
	\partial_t \mathbf p (y,t) = \partial_y \mathbf r (y,t), 
	\quad 
	\partial_t \mathbf e (y,t) = \frac{1}{\overline{m}}  \partial_y \big(\mathbf r (y,t) \mathbf p (y,t)\big) ,
$$
corresponding to Euler equations in Lagrangian coordinates, with $\overline{m}$ the average mass. 

Instances of rigorous derivation of Euler equations in the smooth regime rest on the ergodicity of the microscopic dynamics.
In \cite{eo,ovy,komorowski_olla}, the Hamiltonian dynamics is perturbed by some stochastic noise 
acting in such a way that conserved quantities are not destroyed but that the ergodicity of the dynamics can be established rigorously.  

{One of the main motivations of this work is to show that ergodicity 
is not in general a necessary assumption for Euler equations to hold. Indeed,} 
the dynamics considered here is purely Hamiltonian, non-ergodic, 
and possesses actually a full set of invariant quantities (see Section \ref{subsec: invariant quantities}). 
In the clean case, i.e.\@ when the masses are all equal, we show in Section \ref{subsec: clean} that Euler equations hold if and only if the temperature profile is constant.
Instead, in Section \ref{subsec: disordered}, we argue that Euler equations hold even out of thermal equilibrium if there is disorder on the masses. 
We briefly discuss the fate of other conserved quantities in Section \ref{subsec: other conserved}. 
Theorem \ref{the: main result} in Section \ref{sec: model and results} constitutes the main result of our paper: 
We show the convergence to Euler equations for the disordered harmonic chain, almost surely with respect to the masses and on average with respect to an initial local Gibbs state. 
The rest of the paper is devoted to the proof of this theorem.

\subsection{Clean harmonic chain}\label{subsec: clean}
Let us assume that all masses $m_x$ are equal, say $m_x = 1$ for simplicity. 
In this case, the equations of motion read
\begin{equation}\label{eq:1}
	\dot q_x = p_x , \qquad \dot p_x = \Delta q_x .
\end{equation}
Let us first consider the thermal equilibrium case: 
Assume that the initial configuration of the chain is random and distributed according to a Gaussian law $\mu_0$ with covariance matrix
\begin{equation}\label{eq:2}
	\lgs{(\nabla_+ q)_x ; (\nabla_+ q)_y} =  \lgs{ p_x ; p_y} = \beta^{-1} \delta_{x,y}, \qquad 
  	\lgs{q_x; p_y} = 0,
\end{equation}
for some inverse temperature $\beta$. 
It is easy to prove that at time $t>0$, the distribution $\mu_t$ in the phase space is still given by a Gaussian law with the same covariance matrix.  
To see this, just use Fourier transforms to diagonalize the dynamics:
\begin{equation}\label{eq:3}
  	\hat q(k,t) = \sum_{x\in \ZZ} e^{i 2\pi k x} q_x(t), \qquad  \hat p(k,t) = \sum_{x\in \ZZ} e^{i 2\pi k x} p_x(t),
\end{equation}
and define the wave function 
\begin{equation}\label{eq:4}
  	\hat \phi (k,t) = \omega(k)  \hat q(k,t) + i  \hat p(k,t) 
\end{equation}
where $\omega(k) = |2\sin(\pi k)|$ is the dispersion relation. Then, the explicit solution of \eqref{eq:1} is given by
\begin{equation}
  \label{eq:5}
   \hat \phi (k,t) = e^{-i\omega(k) t} \,  \hat \phi (k,0).
\end{equation}
The correlations \eqref{eq:2} imply that 
\begin{equation}\label{eq:5-0}
	\lgs{ \hat \phi (k,0)^* ;  \hat \phi (k',0)}  = 2 \beta^{-1} \delta(k-k'), \qquad  \lgs{ \hat \phi (k,0) ; \hat \phi (k',0)}=0.
\end{equation}
Consequently
\begin{equation}
\label{eq:6}
\begin{split}
& \lgs{ \hat \phi (k,t)^* ;  \hat \phi (k',t)} =   e^{i(\omega(k)- \omega(k')) t}\,  \lgs{ \hat \phi (k,0)^* ;  \hat \phi (k',0)}
= 2\beta^{-1} \delta(k-k'),\\
&\lgs{ \hat \phi (k,t) ;  \hat \phi (k',t)} =   e^{-i(\omega(k)+\omega(k')) t} \, \lgs{ \hat \phi (k,0) ;  \hat \phi (k',0)}
=0. 
\end{split} 
\end{equation}
From (\ref{eq:5-0}) and (\ref{eq:6}), we deduce that the covariances  \eqref{eq:2} are the same at any time $t$:
\begin{equation}
 \label{eq:2222}
  \lgs{(\nabla_+ q)_x (t) ; (\nabla_+ q)_y (t) } =  \lgs{ p_x (t)  ; p_y (t) } = \beta^{-1} \delta_{x,y}, \qquad 
  \lgs{q_x (t) ; p_y (t) } = 0,
\end{equation}
which implies that the Gaussian distribution $\mu_t$ differs from $\mu_0$ only by the averages $\bar r_x (t) = \lgs{r_x (t)} = \lgs{(\nabla_+ q)_x(t)}$ and $\bar p_x (t) = \lgs{p_x(t)}$ that, 
by linearity of the dynamics, evolve following the same equation \eqref{eq:1}. 

Assume now that the initial averages of the momentum $p_x$ and stretch $r_x = (\nabla_+ q)_x$ are slowly varying on a macroscopic scale. 
{More precisely, let $N$ be an integer representing a macroscopic number of sites in the chain, 
and let $\mathsf p, \mathsf r : \R \to \R$ be smooth and fast decaying initial macroscopic profiles. 
We let}
\begin{equation}
  \label{eq:7}
  {\bar r_{[Ny]}(0) = \mathsf r (y),\qquad
  \bar p_{[Ny]}(0) = \mathsf p (y)} . 
\end{equation}
Let $\widehat {\mathsf p} (\xi)$ and $\widehat {\mathsf r} (\xi)$ be the Fourier transforms (in $\RR$) of $\mathsf p(y)$ and $\mathsf r(y)$. 
Then, as $N\to\infty$, 
$\frac 1N \widehat{\bar p}(\tfrac \xi N) \longrightarrow \widehat {\mathsf p}(\xi)$ and 
$\frac 1N \widehat{\bar r}(\tfrac \xi N) \longrightarrow \widehat {\mathsf r} (\xi)$. After a straightforward analysis 
we have that
\begin{equation}
  \label{eq:8}
  \begin{split}
    \frac 1N \widehat{\bar p}\Big(\tfrac \xi N, Nt\Big) \longrightarrow \widehat {\mathbf p} \, (\xi,t), \qquad 
    \frac 1N \widehat{\bar r} \Big(\tfrac \xi N, Nt \Big)
    \longrightarrow \widehat {\mathbf r}\, (\xi,t )
  \end{split}
\end{equation}
where 
\begin{equation}
  \label{eq:9}
 \partial_t  \widehat {\mathbf r} \, (\xi,t) = -i {2\pi} \xi  \,  \widehat {\mathbf p} \, (\xi,t), \qquad 
 \partial_t  \widehat {\mathbf p} \, (\xi,t) = -i {2\pi} \xi \,  \widehat {\mathbf r} \, (\xi,t).
\end{equation}
Consequently {$\bar r_{[Ny]}(Nt)$ and $\bar p_{[Ny]}(Nt)$} converge {(as distributions)} to the solution of the linear wave equation
\begin{equation}
  \label{eq:10}
  \partial_t  \mathbf r \,(y,t) = \partial_y \mathbf p \, (y,t), \qquad 
 \partial_t \mathbf p \,(y,t) = \partial_y   \mathbf r \, (y,t).
\end{equation}
Let us now consider the energy per particle $e_x = \frac 12\left(p_x^2 + r_x^2\right)$. 
Its average under the distribution $\mu_t$ is 
$\lgs{ e_x (t)} = \beta^{-1} + \frac 12\left(\bar p_x^2(t) + \bar r_x^2(t)\right)$ since by (\ref{eq:2222}), the variance of $p_x$ and $r_x$, i.e. the temperature, remains constant in time. In the limit $N\to \infty$ we have
$$\lgs{ e_{[Ny]}(Nt)} \longrightarrow {\bf e} \, (y,t) = \beta^{-1} + \frac 12\left( {\bf p}^2 (y,t) + {\bf r}^2(y,t)\right),$$
i.e. it solves the equation
\begin{equation}
  \label{eq:11}
   \partial_t {\bf e}\, (y,t) = \partial_y \left( {\bf p}\, (y,t)   \, {\bf r}\, (y,t)\right).
\end{equation}
We recognize that (\ref{eq:10} - \ref{eq:11}) are the Euler equations. 
The above is the simplest example of propagation of local equilibrium and hydrodynamic 
limit in hyperbolic scaling: in a harmonic chain in thermal equilibrium at temperature $\beta^{-1}$, 
and the mechanical modes not in equilibrium, we  will have a persistence of 
the thermal equilibrium {at any time $t$}, 
while the mechanical modes evolve independently from the thermal mode  following the linear wave equation. 

Notice that the argument above does not require the distribution $\mu_0$ to be the thermal equilibrium measure defined by \eqref{eq:2}, 
and that it holds for any measure $\mu_0$ with translation invariant covariance given by
\begin{equation}
  \label{eq:2-c}
  \lgs{\nabla q_x ; \nabla q_y} =  \lgs{ p_x ; p_y} = C(x-y), \qquad 
  \lgs{q_x; p_y} = 0
\end{equation}
for a positive definite function $C(x)$. 
The only difference is then that in (\ref{eq:5-0}- \ref{eq:6}) the term $\beta^{-1}$ has to be replaced by the Fourier transform ${\widehat C} (k)$. 
Actually, the measure $\mu_t$ is not even a local equilibrium state \cite{BO_book}, underlining that the validity of Euler equations in this example does not require the propagation of local equilibrium.

The above argument rests on the translation invariance of the distribution of the thermal modes and fails if it is space inhomogeneous, 
for example if the starting distribution is given by a local Gibbs state with a slowly varying temperature $\beta_{N,x}^{-1} = \beta^{-1}(x/N)$, 
i.e.\@ a Gaussian measure with covariances
\begin{equation}
  \label{eq:12}
  \lgs{\nabla q_x ; \nabla q_y} =  \lgs{ p_x ; p_y} = \beta_{N,x}^{-1}\delta_{x,y}, \qquad 
  \lgs{q_x; p_y} = 0.
\end{equation}
In this case, even though the wave equation \eqref{eq:10} still holds, 
generally the energy 
equation \eqref{eq:11} is not valid. 
In fact the energies of each mode $k$ evolves autonomously, 
as we can see studying the limit evolution of the Wigner distribution defined by
\begin{equation}\label{eq:wigner1}
\begin{split}
    	\widehat W_N(\xi, k,t) &:= \frac 2N 
    	\Blgs{\widehat\phi^*\left(k - \tfrac{\xi}{2N}, Nt\right)
       	\widehat\phi\left(k + \tfrac{\xi}{2N}, Nt\right)}\\
   	 W_N(y, k,t) &:= \int e^{-i 2\pi \xi y}   \widehat W_N (\xi, k,t) \; d\xi,
\end{split}
\end{equation}
(the above definitions should be understood as distributions on $\R\times \Pi$ with $\Pi = \R\backslash \Z$).

In the limit as $N\to\infty$ the Wigner distribution converge to a positive distribution
with an absolutely continuous part, the local distribution of the thermal modes, and 
a singular part concentrated on $k=0$, the mechanical modes:
\begin{equation}
  \label{eq:18}
  \lim_{N\to\infty} \widehat W_N(\xi, k, t) = \widehat W_{th}(\xi, k,t) + \widehat W_m(\xi, t)\;  \delta_0(dk) 
\end{equation}
The mechanical part $\widehat W_m(\xi, t)$ is the Fourier transform of 
$\frac 12\left( {\bf p}^2 (y,t) + {\bf r}^2(y,t)\right)$.

A straightforward calculation gives for the thermal part (see \cite{dobru} or \cite{bos} for a rigorous argument):
\begin{equation}
  \label{eq:14}
  \widehat W_{th}(\xi, k,t) = e^{-i\omega'(k)\xi t} \; \widehat W_{th}(\xi, k,0).
\end{equation}
This implies that the inverse Fourier transform  $W_{th}(y, k,t)$ satisfies the transport equation
\begin{equation}
  \label{eq:13}
  \partial_t W_{th}(y, k,t) + \frac{\omega'(k)}{2\pi} \partial_y W_{th}(y, k,t) = 0.  
\end{equation}
It also follow that 
\begin{equation}
  \label{eq:15}
  \int  W_{th}(y, k,t) \; dk \ = \ \tilde {\bf e}(y,t)
\end{equation}
where $\tilde {\bf e}(y,t)$ is the limit profile of thermal energy (or temperature) defined as
\begin{equation}
  \label{eq:16}
  \frac 12 \left(\lgs{ r_{[Ny]}(Nt); r_{[Ny]}(Nt)} + \lgs{ p_{[Ny]}(Nt); p_{[Ny]}(Nt)}\right) \rightharpoonup \tilde {\bf e}(y,t).
\end{equation}
Consequently the thermal energy  $\tilde {\bf e}(y,t)$ evolves non autonomously following the equation
\begin{equation}
  \label{eq:17}
  \partial_t \tilde {\bf e}(y,t) + \partial_y J(y,t) = 0, \qquad J(y,t) = \int \omega'(k)  W_{th}(y, k,t) \; dk.
\end{equation}
We say that the system is in {\it local equilibrium} if $W_{th}(y,k) = \beta^{-1}(y)$ constant in $k$. 
This correspond to the fact that Gibbs measure gives uniform distribution on the modes.
Starting in thermal equilibrium means $W_{th}(y,k,0) = \beta^{-1}$ and trivially $W_{th}(y,k,t) = \beta^{-1}$ 
for any $t>0$.
But starting with local equilibrium, i.e. $W(y,k,0) = \beta^{-1}(y)$ constant in $k$, 
we have a non autonomous evolution of $\tilde e(y,t)$.

\subsection{Disordered harmonic chain}\label{subsec: disordered}
The situation so far can be summarized as follows. 
By linearity, the variables $r_x$ and $p_x$ admit a macroscopic limit described by \eqref{eq:10} independently of the initial temperature profile. 
The macroscopic equation \eqref{eq:11} predicts that the evolution of the energy is purely mechanical and that the temperature does not evolve with time. 
As it turns out, the evolution of the mechanical energy is correctly described by Euler equation (see the term $\mathcal A_N (t)$ in our decomposition \eqref{eq: average and fluctuation} below), 
but thermal fluctuations do in general evolve with time as well, except if the temperature profile is initially flat. 

This picture gets strongly modified if the masses are taken to be random. 
On the one hand, deriving the macroscopic evolution of the fields $r_x$ and $p_x$ becomes less obvious because some homogenization over the masses is required. 
This difficulty can be solved by the elegant method of the 
``corrected empirical measure'', see \cite{goncalves_jara,jara_landim,bernardin} 
(though we will actually solve it another way). 
On the other hand, and this is the main point in considering random masses, 
the evolution of the energy $e_x$ is now much better approximated by Euler equation.
Indeed, at a microscopic level, all thermal fluctuations are frozen thanks 
to Anderson localization and the evolution of the energy becomes purely mechanical. 

To understand this a little bit better, it is good to realize how the disorder modifies the nature 
of the eigenmodes $(\psi^k)_{1 \le k \le N}$ of the operator $M^{-1}\Delta$ for a finite chain of size $N$. 
As a consequence of Anderson localization \cite{anderson}, all modes at positive energy are spatially localized. 
However the localization length $\zeta_k$ diverges as one approaches the ground state: 
$$
	\zeta_k^{-1} \sim \omega^2_k \sim \Big(\frac{k}{N}\Big)^2,
$$
so that only the modes with $k\gtrsim \sqrt N$ are actually localized, while the modes $k \lesssim \sqrt N$ remain comparable to the modes of the clean chain \cite{verheggen,ajanki_huveneers}.
By imposing a smooth initial profile $\mathsf r, \mathsf p$, the initial local Gibbs state attributes a weight of order 1 to a few first modes above the ground state, and a weight of order 1 to all other modes together. 
The first ones are responsible for the transport of mechanical energy; 
all modes with $k \gg \sqrt N$ are localized and do not transport any thermal energy;  
all modes with $1  \ll k \le o (N)$ have a vanishing weight in the thermodynamic limit and can be neglected in the analysis. 

Finally, we would like to mention that, while the disorder considered here 
and the stochastic velocity exchange noise considered in \cite{eo,komorowski_olla} 
act in an obviously very different way, 
e.g.\@ the disorder preserves integrability while the stochastic noise makes the dynamic ergodic, 
they do produce the same effects in some respect. 
{Indeed the noise has only a very slow (negligeable) effect on the macroscopic modes.}
% $(r_x,p_x)$ profiles that vary smoothly with $x$ at the macroscopic scale (because nearly identical momenta are exchanged). 
{ This bares some similarity with the fact that the disorder has very little influence on the low modes of the disordered chain,
while the dynamical noise provides an active hopping mechanism among the high modes. 
Consequently the dynamical noise produces a superdiffusive sub-ballistic spreading of the thermal energy \cite{jara_kom_olla,komorowski_olla}, 
that is not visible in the hyperbolic scaling.
Thus, the dynamical noise plays here as well a role analogous to the disorder only in the hyperbolic scaling by freezing the temperature profile.

It is important to notice that in the disordered case the temperature profile remains frozen at any time scale, 
including the diffusive time scale and further, see remark \ref{frozenforever}.
In particular this implies a vanishing thermal diffusivity for the disordered unpinned harmonic chain. This is not in contradiction with the divergence of the thermal conductivity observed in the NESS of the same unpinned system when connected to Langevin thermostats at different temperatures with free boundary conditions, see \cite{casher_lebowitz,ajanki_huveneers}, and our result sheds actually some light on the various behaviors for the conductivity found in \cite{dhar} depending on the boundary conditions. In fact the thermal conductivity divergence in the NESS is due to the fluctuations of the low mechanical modes, and in our analysis there is a clear separation of the behavior of the mechanical modes (responsable for the ballistic motion) and the high thermal modes that give the temperature profile.}

\subsection{Other conserved quantities}\label{subsec: other conserved}
Before moving on, let us briefly comment on the issue of the other conserved quantities of the system. 
These can also be written as a sum of local terms and lead thus to additional conservation laws. 
For example, 
$$
	I \; = \; \sum_{x} d_x \; = \; \frac{1}{2}\sum_x\left( \frac{(r_x - r_{x-1})^2}{m_x} + \left( \frac{p_{x+1}}{m_{x+1}} - \frac{p_x}{m_x} \right)^2 \right)
$$
is conserved (see Sections \ref{subsec: a priori} and \ref{subsec: invariant quantities}) and leads to the microscopic conservation law
$$
	\dot d_x = \left( \frac{p_{x+1}}{m_{x+1}} - \frac{p_{x}}{m_{x}} \right) \frac{r_{x+1} - r_x}{m_{x+1}} -  \left( \frac{p_{x}}{m_{x}} - \frac{p_{x-1}}{m_{x-1}} \right) \frac{r_{x} - r_{x-1}}{m_x} .
$$

It is thus natural to ask whether this relation generates also some macroscopic law. 
In the cases where we can derive the macroscopic evolution equation \eqref{eq:11} for the energy, it is easy to argue that the corresponding macroscopic density $\mathbf d (y,t)$ does not evolve with time in the hyperbolic scaling.
Indeed, we can decompose $d_x$ as the sum of a mechanical and a thermal contribution, as we do in \eqref{eq: average and fluctuation} below for the energy.
In this case, contrary to what happens for the energy, the mechanical contribution vanishes in the thermodynamic limit since $d_x$ depends on $r$ and $p$ only through their gradients, 
while the contribution from the thermal modes does not evolve with time, for the same reasons as it does not for the energy.  

All the other conserved quantities in this model that can be written as a sum of local terms are obtained by taking further gradients in the variables $r$ and $p$ (see Section \ref{subsec: invariant quantities}), 
and have thus no evolution either in the hyperbolic scaling.

%section
\section{Model and results}\label{sec: model and results}

We define the model studied in this paper and we state our main result. 
For technical reasons, it is easier to work on a finite system of size $N$ and then let $N\to \infty$.

\subsection{Hamiltonian model}
The Hamiltonian $H$ on $\R^{2N}$ is defined by 
$$
	H (q,p) = \frac{1}{2} \sum_{x=1}^N \left( \frac{p_x^2}{m_x} + ((\nabla_+q)_{x})^2 \right).
$$
{For concreteness, we assume} free boundary conditions, i.e.\@ $q_0 = q_1$ and $q_{N+1} = q_N$; 
{other boundary conditions such as fixed or periodic could be considered just as well.} 
The masses $(m_x)_{1 \le x \le N}$ are i.i.d.\@ random variables. 
In order to avoid any technical difficulty in exploiting known results from the Anderson localization literature, we assume that the law of $m_x$ admits a smooth density compactly supported in $[m_-,m_+]$ with $m_->0$. 

The equations of motion read $M \dot q = p$ and $\dot p = \Delta q$ where $M$ is the square diagonal matrix of size $N$ with entries defined by $M_{x,y} = \delta (x-y) m_x$ ($\delta(z)$ is defined by $1$ for $z=0$ and $0$ otherwise).
It is more convenient to express the equations of motion in terms of the displacement variables 
\begin{equation*}
	r_x  =  (\nabla_+ q)_x \qquad (1 \le x \le N-1) .  
\end{equation*}
The equations of motion become
\begin{equation}\label{eq: equations of motion}
	\dot r_x = \big( \nabla_+ M^{-1}p \big)_x \quad (1 \le x \le N-1), 
	\qquad \dot p_x = (\nabla_- r)_x \quad (1 \le x \le N)
\end{equation}
where we use fixed boundary conditions for $r$ in the second equation: $r_0 = r_N = 0$.

\subsection{Gibbs and locally Gibbs states}
We consider three locally conserved quantities in the bulk: 
$$ H = \sum_{x=1}^N e_x =  \sum_{x=1}^N \left( \frac{p^2_x}{2m_x} + \frac{r_x^2}{2} \right), \qquad P = \sum_{x=1}^N p_x, \qquad R = \sum_{x=1}^{N-1} r_x. $$
The energy $H$ and the momentum $P$ are actually truly conserved, but the conservation of $R$ is broken at the boundary: $\dot R = m_N^{-1}p_N - m_1^{-1}p_1$. 

The Gibbs states are characterized by three parameters: 
$\beta>0$ and $\mathsf p, \mathsf r \in \R$. Its probability density writes
$$ 
\rho_{\text{G}} (r,p) = \frac{1}{Z_{\text{G}}} \exp \Big\{- \frac{\beta}{2} \sum_{x=1}^{N} m_x\Big(\frac{p_x}{m_x} - \frac{\mathsf p}{\overline m}\Big)^2 - \frac{\beta}{2} \sum_{x=1}^{N-1} (r_x - \mathsf r)^2  \Big\} .
$$
where $\overline m$ denotes the mean mass and $Z_{\text{G}} :=Z_{\text{G}} (\beta, \mathsf p, \mathsf r)$ is a normalizing constant.
Local Gibbs states are obtained by replacing the constant parameters $\beta, \mathsf p, \mathsf r$ by functions 
$$\beta,\mathsf p, \mathsf r : [0,1] \to \R,$$  
with $\beta(x)>0$ for all $x \in [0,1]$, and by considering the measure with density 
\begin{equation}\label{eq: local Gibbs}
	\rho_{\text{loc}} (r,p) = \frac{1}{Z_{\text{loc}}} \exp \Big\{- \frac{1}{2} \sum_{x=1}^{N} \beta(x/N) m_x\Big(\frac{p_x}{m_x} - \frac{\mathsf p(x/N)}{\overline{m}}\Big)^2 - \frac{1}{2} \sum_{x=1}^{N-1} \beta(x/N) (r_x - \mathsf r(x/N))^2  \Big\}
\end{equation}
where $Z_{\text{loc}}:=Z_{\text{loc}} (\beta, \mathsf p, \mathsf r)$ is a normalizing constant.
We impose the following regularity conditions on $\beta, \mathsf p, \mathsf r$: 
\begin{equation}\label{eq: regularity beta r p}
	\beta \in \mathcal C^0([0,1]), 
	\quad \mathsf r \in \mathcal C^1([0,1]) \text{ with }\mathsf r(0) = \mathsf r(1) = 0, 
	\quad \mathsf p \in \mathcal C^1([0,1]).
\end{equation}
We take such a local Gibbs state as initial state.
Below, we denote the expectation with respect to it by $\lgs{\cdot}$:
$$
	\lgs{F} = \int F(r,p) \rho_{\text{loc}} (r, p) \, \dd r \dd p.
$$ 
Instead, expectation (resp.\@ probability) with respect to the masses is denoted by $\mathsf E$ (resp.\@ $\mathsf P$).

\subsection{Evolution of the locally conserved quantities} 
Let us fix some maximal time $T>0$. 
Let us define the fields $\mathcal R$, $\mathcal P$ and $\mathcal E$ acting on functions $f \in \mathcal C^0([0,1])$ as 
\begin{equation}\label{eq:rpe}
\begin{split}
&\mathcal R(f,t) = \int_0^1 \mathbf r (y,t)\,  f(y) \, \dd y, \quad  \mathcal P(f,t) = \int_0^1 \mathbf p (y,t) \, f(y)\,  \dd y,\\
&\mathcal E(f,t) = \int_0^1 \mathbf e (y,t) \, f(y) \, \dd y. 
\end{split}
\end{equation}
for all $t \in [0,T]$. 
The kernels $\mathbf r$, $\mathbf p$ and $\mathbf e$ are defined as follows. 
First, at $t=0$, we impose
$$
	\mathbf r(y,0) = \mathsf r(y), \qquad 
	\mathbf p(y,0) = \mathsf p(y), \qquad
	\mathbf e(y,0) =  \frac1{\beta(y)} +  \frac{\mathsf p^2 (y)}{2\overline m} + \frac{\mathsf r^2 (y)}{2}. 
$$
Next, the evolution at all further time is governed by the following system of conservation laws: 
\begin{align}
	&\partial_t \mathbf r(y,t) = \frac1{\overline m} \partial_y \mathbf p(y,t), \qquad \mathbf r(0,t) = \mathbf r(1,t) = 0, 
	\label{eq:equa diff r}\\
	&\partial_t \mathbf p(y,t) = \partial_y \mathbf r (y,t), % \corc{\qquad \partial_x \mathbf p(0,t) = \partial_x \mathbf p(1,t) = 0},
	\label{eq:equa diff v}\\
	&\partial_t \mathbf e(y,t) = \frac1{\overline m} \partial_y (\mathbf r(y,t) \mathbf p(y,t)).
	\label{eq:equa diff e}
\end{align}

Thanks to the regularity conditions on $\mathsf r, \mathsf p$ in \eqref{eq: regularity beta r p}, the solutions of these equations are classical.
Since $(\mathbf r, \mathbf p)$ are solution of wave equations with suitable boundary conditions, they can be obtained explicitly by expanding them in Fourier series.  
Then, by a time integration, $\mathbf e$ may be expressed as a function of $(\mathbf r, \mathbf p)$, see (\ref{eq: evolution energy limit}). 
Later we will use that a classical solution for the system governing $(\mathbf r, \mathbf p)$ coincides with the (unique) weak solution of this system. 
Because of the boundary conditions, test functions will have to be chosen appropriately (see (\ref{eq: R characterization}-\ref{eq: P characterization})). 

\begin{Theorem}\label{the: main result}
Let $t\in [0,T]$ and $f \in \mathcal C^0([0,1])$. 
Let us assume that the system is initially prepared in a locally Gibbs state such that $\beta$, $\mathsf r$ and $\mathsf p$ satisfy \eqref{eq: regularity beta r p}. 
Then, as $N\to \infty$, almost surely (w.r.t.\@ $\mathsf P$), 
\begin{align}
	&\mathcal R_N (f,t) \; = \;  \frac{1}{N} \sum_{x=1}^N f(x/N) \; \lgs{r_x (Nt)} \qquad \to \qquad \mathcal R(f,t), \label{R limit}\\
	&\mathcal P_N (f,t) \; = \; \frac{1}{N} \sum_{x=1}^N f(x/N) \; \lgs{p_x (Nt)} \qquad \to \qquad \mathcal P(f,t), \label{P limit}\\
	&\mathcal E_N (f,t) \; = \; \frac{1}{N} \sum_{x=1}^N f(x/N) \; \lgs{e_x (Nt)} \qquad \to \qquad \mathcal E(f,t) \label{E limit}.
\end{align}
\end{Theorem}
\begin{Remark}
As pointed out in the introduction, the situation is much simpler at thermal equilibrium, i.e.\@ for $\beta$ constant, and these limits hold even for the non-disordered chain.
See Section \ref{subsec: thermal equilibrium} for a derivation along the lines used to derive Theorem \ref{the: main result}. 
\end{Remark}

%section
\section{Evolution of $\mathcal R_N$ and $\mathcal P_N$}

In this section, we show the limits (\ref{R limit}-\ref{P limit}). 
Moreover, in order to later deal with the field $\mathcal E_N$, we show more: 

The functions $\lgs{r_{[Ny]} (Nt)}$ and $m_{[Ny]}^{-1}\lgs{ p_{[Ny]} (Nt)}$ are uniformly (in $N$) H\"older regular in $y\in [0,1]$, with exponent at least $1/2$. 
Hence they converge pointwise to $\mathbf r(y,t)$ and $\mathbf p (y,t)$ respectively. 

\subsection{A priori estimates}\label{subsec: a priori}
Given $d\in \N$, we denote the standard scalar product on $\R^d$ by $\langle \cdot , \cdot \rangle_d$ (we will drop the subscript $d$ when no confusion seems possible).
Let us consider the two following conserved quantities: 
\begin{align}
	H(r,p) &= \frac12 \big(  \langle p ,M^{-1 } p \rangle_N + \langle r ,r \rangle_{N-1} \big), \\
	I(r,p) &= \frac12 \big( \langle \nabla_- r, M^{-1} \nabla_- r \rangle_N + \langle \nabla_+ M^{-1} p , \nabla_+ M^{-1}p \rangle_{N-1} \big).\label{eq: I conserved}
\end{align}
The conservation of $I$ follows from the fact that, if $(r,p)$ solve \eqref{eq: equations of motion}, then $(\nabla_+ M^{-1}p, \nabla_- r)$ solve the same equation, the corresponding Hamiltonian being $I$ 
(since we have that $H (\nabla_+ M^{-1} p, \nabla_- r) =I (r,p)$).
Notice also that a full set of conserved quantities can be generated by further taking gradients, see Section \ref{subsec: invariant quantities}. 

Thanks to these two conservation laws, and to the smoothness assumptions on $\mathsf r$ and $\mathsf p$, we deduce
\begin{Lemma}\label{lem-bounds}
There exists {a deterministic} $\Const$ such that, for any $t \ge 0$ and any $N \in \N$, 
\begin{align} 
	&\sum_{x=1}^{N-1} \lgs{ r_x(Nt) }^{2} \le \Const N, \qquad 
	\sum_{x=1}^N \lgs{ p_x(Nt) }^2 \le \Const N, 
	\label{eq: L2 estimate}\\
	&\sum_{x=1}^N \lgs{ (\nabla_- r)_x (Nt)}^2 \le \frac\Const{N}, \qquad 
	\sum_{x=1}^{N-1} \lgs{ (\nabla_+ M^{-1} p)_x (Nt) }^2 \le \frac\Const{N}.
	\label{eq: H1 estimate}
\end{align}
\end{Lemma}
\begin{proof}
By linearity of the equations of motion \eqref{eq: equations of motion}, $(\lgs{ r }, \lgs{ p } )$ solve the same equations as $(r,p)$. 
Therefore, the conservation of $H(r,p)$ and $I(r,p)$ implies the conservation of $H(\lgs{ r }, \lgs{ p })$ and $I(\lgs{ r }, \lgs{ p })$. 
Since the quantities to be estimated in \eqref{eq: L2 estimate} are bounded by $H(\lgs{ r }, \lgs{ p })$ 
and the quantities to be estimated in \eqref{eq: H1 estimate} are bounded by $I(\lgs{ r }, \lgs{ p })$, 
we conclude that is it enough to establish them respectively for $H(\lgs{r}, \lgs{p})$ and  $I(\lgs{r}, \lgs{p})$ at $t=0$. 
This follows from a direct computation, thanks to the product structure of the local Gibbs state \eqref{eq: local Gibbs} and to the hypotheses on $\mathsf r$ and $\mathsf p$ in \eqref{eq: regularity beta r p} 
(in particular, this is the place where the boundary condition on $\mathsf r$ plays a role). 
\end{proof}
\noindent

{
\begin{Remark}
Notice that the bounds in Lemma \ref{lem-bounds} are actually valid for any time scale $N^\alpha t$, for any $\alpha>0$.
\end{Remark}
}

As a corollary, we deduce the existence of a constant $\Const\in \R$ such that, for any $x,y \in \Z \cap [1,N]$,
\begin{equation}
\label{eq: Holder continuity}
\begin{split}
&\big| \lgs{r_{x'}(Nt) } - \lgs{r_x(Nt) } \big| \le  \Const \left|\tfrac{x'-x}N \right|^{1/2},\\
&\big| m_{x'}^{-1} \lgs{ p_{x'}(Nt) } -  m_x^{-1} \lgs{ p_x(Nt) } \big|  \le  \Const \Big|\tfrac{x'-x}N \Big|^{1/2},
\end{split}
\end{equation}
and therefore also such that
\begin{equation}\label{eq: bounded r and p}
	|\lgs{ r_x(Nt) }| \le \Const, \quad |\lgs{ p_x(Nt) }| \le \Const .
\end{equation}
Indeed, to get e.g.\@ \eqref{eq: Holder continuity} for $r$, we deduce from \eqref{eq: H1 estimate} that 
\begin{equation*}
\begin{split}
\big|\lgs{ r_{x'}(Nt) } - \lgs{ r_x(Nt) } \big|&=\left|\sum_{z=x+1}^{x'} \lgs{ (\nabla_- r)_z(Nt)} \right| \le \left( \sum_{z=1}^N \lgs{ (\nabla_- r)_z (Nt) }^2 \right)^{1/2} |x-x'|^{1/2} \\
&\le \Const \Big|\tfrac{x'-x}N \Big|^{1/2}.
\end{split}
\end{equation*}
{Next \eqref{eq: bounded r and p} follows from \eqref{eq: Holder continuity} if, given $N,t$, there exists at least some $x_0$ such that the inequalities hold.
This in turn follows from \eqref{eq: L2 estimate}.}

\subsection{Averaging lemma for the field $\mathcal P_N$}
The method of the corrected empirical measure is an elegant method to deal with the randomness on the masses in deriving the hydrodynamic limit for $\mathcal R_N$ and $\mathcal P_N$ \cite{goncalves_jara,jara_landim,bernardin}.
However, in our case, it seems more convenient to use the following lemma: 

\begin{Lemma}\label{lem: replacement}
Let $f \in \mathcal C^0([0,1])$ and $t\ge 0$. 
Almost surely (w.r.t.\@ the masses), for $N\to\infty$, 
\begin{align}
	 &\frac1N \sum_{x=1}^N f (x/N) \frac{\lgs{ p_x (Nt) }}{m_{x}} (m_x - \overline m) \quad \to \quad 0,
	 \label{eq: replacement}\\
	 &\frac1N \sum_{x=1}^N f (x/N) \left(\frac{\lgs{ p_x (Nt) }}{m_{x}}\right)^2 (m_x - \overline m) \quad \to \quad 0.
	 \label{eq: replacement bis}
\end{align}
\end{Lemma}

\begin{proof}
Let us start with \eqref{eq: replacement}.
Let $A_N$ be the quantity in the left hand side of \eqref{eq: replacement},
let $\widetilde m_x = m_x - \overline m$, and let 
\begin{equation}\label{eq: some def of varphi}
	\varphi(x) =  f(x/N) \frac{\lgs{ p_x (Nt) }}{m_{x}}
\end{equation}
{where, for simplicity, we do not write explicitly the dependence of $\varphi$ on $N$ and $t$}.
Let $0 < \tau < 1$. 
{Let}
$$
	{\Gamma_N^\tau = \{ 1 + k \lfloor N^\tau \rfloor : k \in \Z \} \cap [1,N] } 
$$ 
{and, given $x\in \Gamma_N^\tau$, let }
$$
	{ x' = \min \{ y \in \Gamma_N^\tau \cup \{ N+1\} : y > x \} .} 
$$
We decompose $A_N$ as
\begin{align*}
	A_N
	=& {\frac{1}{N^{1 - \tau}} \sum_{x \in \Gamma_N^\tau} \frac{1}{N^\tau} \sum_{x \le z < x'} \varphi (z) \tilde m_z }\\
	=&{ \frac{1}{N^{1 - \tau}} \sum_{x \in \Gamma_N^\tau} \frac{\varphi (x)}{N^\tau} \sum_{x \le z < x'} \tilde m_z  
	 + \frac{1}{N^{1 - \tau}} \sum_{x \in \Gamma_N^\tau} \frac{1}{N^\tau} \sum_{x \le z < x'} (\varphi (z) - \varphi (x)) \tilde m_z }\\
	=: & A_N^{(1)} + A_N^{(2)}.
\end{align*}
To deal with $A_N^{(1)}$, we observe that $\varphi$ is bounded, see \eqref{eq: bounded r and p}, so that by Jensen's inequality, 
$$
	\mathsf E((A_N^{(1)})^4) 
	\le {\frac{\Const}{N^{1-\tau}} \sum_{x \in \Gamma_N^\tau} \mathsf E \bigg(\bigg( \frac{1}{N^\tau} \sum_{x\le z <x'} \widetilde m _{z}  \bigg)^4\bigg) }
	\le \frac{\Const}{N^{2 \tau}}
$$
{(since the masses $m_x$ are i.i.d.).} 
Taking $\tau > 1/2$, this shows that $A_N^{(1)} \to 0$ almost surely by Borel-Cantelli's lemma.
{To deal with $A_N^{(2)}$, we start from the definition \eqref{eq: some def of varphi} of $\varphi$ and we bound}
\begin{equation}\label{eq: bound varphi in proof lemma 2}
	{
	|\varphi (z) - \varphi(x)| 
	\le 
	\Const \left| m_z^{-1}\lgs{ p_z (Nt) } -  m_x^{-1}\lgs{ p_x (Nt) } \right|
	+ \Const |f(z/N) - f(x/N)|.
	}
\end{equation}
{The fact that $A_N^{(2)}\to 0$ (deterministically) follows then from the bound \eqref{eq: Holder continuity} and from the uniform continuity of $f$. }

{The proof of \eqref{eq: replacement bis} is entirely analogous, with now 
$\varphi (x) = f(x/N) (\lgs{ p_x (Nt) } / m_{x} )^2$ instead of \eqref{eq: some def of varphi}. 
By \eqref{eq: bounded r and p}, this function is bounded and \eqref{eq: bound varphi in proof lemma 2} is still satisfied since
\begin{align*}
	&\left| (m_z^{-1}\lgs{ p_z (Nt) })^2 -  (m_x^{-1}\lgs{ p_x (Nt) })^2 \right|\\
	&\le 
	\left| m_z^{-1}\lgs{ p_z (Nt) } +  m_x^{-1}\lgs{ p_x (Nt) } \right| \times \left| m_z^{-1}\lgs{ p_z (Nt) } -  m_x^{-1}\lgs{ p_x (Nt) } \right|\\
	&\le 
	\Const \left| m_z^{-1}\lgs{ p_z (Nt) } -  m_x^{-1}\lgs{ p_x (Nt) } \right|.
\end{align*}
}
\end{proof}

\subsection{Proof of  the convergence to the linear wave equation (\ref{R limit}-\ref{P limit})}\label{sec:proof-}
For any smooth functions $f, g:[0,1] \in \R$ such that $f(0)= f(1) = 0$, 
the limiting fields $\mathcal R$ and $\mathcal P$ defined in (\ref{eq:rpe}) 
can be equivalently characterized as follows: 
\begin{align}
	\mathcal R (f,t) &= \mathcal R (f,0) - \frac{1}{\overline m} \int_0^t \mathcal P (f',s) \dd s, 
	\label{eq: R characterization}\\
	\mathcal P (g,t) &= \mathcal P (g,0) - \int_0^t \mathcal R (g',s) \dd s,
	\label{eq: P characterization}
\end{align}
and 
\begin{equation}\label{eq: limiting fields t=0}
	\mathcal R (f,0) = \int_0^1 f(x) \mathsf r(x) \dd x, \qquad \mathcal P(g,0) = \int_0^1 g(x) \mathsf p (x) \dd x.
\end{equation}
Let us use this characterization to show that $\mathcal R_N (f,t) \to \mathcal R(f,t)$ 
and  $\mathcal P_N (g,t) \to \mathcal P(g,t)$.

The convergence at $t=0$ follows from the strong law of large numbers: $\mathcal R_N (f,0)$ and $\mathcal P_N (g,0)$ converge almost surely to $\mathcal R (f,0)$ and $ \mathcal P(g,0)$ given by \eqref{eq: limiting fields t=0}. 

Let us next consider $t\ge 0$, and let us first deal with $\mathcal R_N$. 
Integrating the equations of motion yields 
\begin{align*}
	\mathcal R_N(f,t) 
	&= \mathcal R_N (f ,0) + \int_0^{Nt}   \frac1{N}\sum_{x=1}^N f(x/N)\, \blgs{ \big( \nabla_+ M^{-1} p \big)_x (s)} \,\dd s \\
	&= \mathcal R_N (f, 0) - \int_0^{Nt}   \frac1{N}\sum_{x=1}^N \nabla_- f(x/N) m_x^{-1}  \lgs{ p_x (s) } \,\dd s
\end{align*}
where we used the boundary condition $f(0) = f (1) = 0$ to perform the integration by part. 
Since $\nabla_- f(x/N) = N^{-1} f'(x/N) + \mathcal O (N^{-2})$, we obtain
$$
	\mathcal R_N (f,t) =  \mathcal R_N (f, 0) - \int_0^{t}   \frac1{N}\sum_{x=1}^N f'(x/N) m_x^{-1}  \lgs{ p_x (Ns) } \dd s + \mathcal O \Big(\frac1N \Big).
$$
{Using \eqref{eq: replacement} in Lemma \ref{lem: replacement}, as well as the dominated convergence theorem to deal with the time integral, 
we may replace} $m_x^{-1}$ by $(\overline m)^{-1}$ up to an error that vanishes almost surely in the limit $N\to \infty$.
Thus 
\begin{equation}\label{eq: RN relation}
	\mathcal R_N (f,t) =  \mathcal R_N (f, 0) -  \frac{1}{\overline{m}}\int_0^{t} \mathcal P_N (f',s) \dd s + \varepsilon_N,
\end{equation}
where $\varepsilon_N \to 0$ almost surely as $N \to \infty$. 
Let us next deal with $\mathcal P_N$. This case is simpler since no homogenization over the masses is needed. Proceeding similarly, we find
\begin{align}
	\mathcal P_N(g,t)
	&= \mathcal P_N (g,0) + \int_0^{Nt} \frac1N \sum_{x=1}^N g(x/N)\, \blgs{\big(\nabla_- r\big)_x(s)}\, \dd s 
	\nonumber\\
	&= \mathcal P_N (g,0) - \int_0^{Nt} \frac1N \sum_{x=1}^N \nabla_+ g(x/N) \lgs{ r_x (s)}\, \dd s 
	\nonumber\\
	& = \mathcal P_N (g,0) - \int_0^t \mathcal R_N (g',s) \dd s + \tilde\varepsilon_N
	\label{eq: PN relation}
\end{align}
where we used the boundary condition $r_0(s) = r_N(s) = 0$ for all time $s\ge 0$ to perform the integration by part, and where $\tilde\varepsilon_N \to 0$ deterministically as $N\to \infty$. 

The families $\big( \mathcal R_N (f,\cdot)\big)_N$ and $\big( \mathcal P_N (g,\cdot)\big)_N$ are equicontinuous since a uniform bound on the time derivative of $\mathcal R_N(f,\cdot)$ and $ \mathcal P_N (g,\cdot)$ holds. 
Hence, the relations \eqref{eq: RN relation} and \eqref{eq: PN relation} implies that any limiting point must satisfy (\ref{eq: R characterization}-\ref{eq: P characterization}).

\subsection{Pointwise convergence}\label{sec:poit}
Thanks to the H\"older regularity
of both $\lgs{ r_{x}(Nt) }$ and $m_x^{-1}\lgs{ p_{x}(Nt) }$ expressed by \eqref{eq: Holder continuity}, 
we deduce a stronger result: 
\begin{Proposition}\label{prop: pointwise convergence}
Let $y\in ]0,1[$ and let $t\in ]0,1[$. 
As $N\to \infty$, almost surely (w.r.t.\@ $\mathsf P$),  
$$
	\lgs{ r_{[Ny]}(Nt) } \quad\to\quad \mathbf r(y,t), 
	\qquad 
	\frac{\lgs{ p_{[Ny]} (Nt) }}{m_{[Ny]}} \quad\to\quad \frac{\mathbf p(y,t)}{\overline{m}}.
$$
\end{Proposition}
\begin{proof}
Let us first deal with $\lgs{ r_{[Ny]}(Nt) }$.
Let $(\rho_\epsilon)_{\epsilon > 0}$ be a regularizing family: 
$\rho_\epsilon \in \mathcal C^\infty (\R)$, 
$\supp (\rho_\epsilon) = [-\epsilon, \epsilon]$, $\rho_\epsilon \ge 0$ and $\int \rho_\epsilon (y) \, \dd y = 1$.
For $y\in ]\epsilon, 1-\epsilon[$, we decompose
\begin{equation*}
	\begin{split}
    	&\lgs{ r_{[Ny]} (Nt) } = \int \rho_\epsilon \Big( y-y' \Big) \lgs{ r_{[Ny]} (Nt) } \dd y'\\
	&=  \int \rho_\epsilon \Big( y-y' \Big) \lgs{ r_{[Ny']} (Nt) } \dd y' + \int \rho_\epsilon \Big( y-y' \Big) \left( \lgs{r_{[Ny]} (Nt)} -\lgs{ r_{[Ny']} (Nt) } \right)\dd y'.
  	\end{split}
\end{equation*}
By \eqref{eq: Holder continuity}, the second term is bounded in absolute value by
\begin{equation}	\label{eq: Holder rest}
  \int \rho_\epsilon \Big( y-y' \Big) \left| \lgs{r_{[Ny]} (Nt)} -\lgs{ r_{[Ny']} (Nt) } \right|\dd y' \ \le\
  \Const \sqrt\epsilon,
\end{equation}
while the first term is approximated uniformly in $N$ by
\begin{equation*}
 \frac{1}N \sum_{x=1}^N \rho_\epsilon \Big( \frac{x}N - y \Big) \lgs{ r_x (Nt) }.
\end{equation*}
Thus, by the result shown {in Section~\ref{sec:proof-}}, this term converges to 
\begin{equation}\label{eq: regularized integral}
	\int_0^1 \rho_\epsilon (y - y') \mathbf r(y',t) \dd y'
\end{equation}
as $N\to \infty$. 
Letting next $\epsilon\to 0$,  the continuity of $\mathbf r(\cdot , t)$ 
implies that \eqref{eq: regularized integral} converges to $\mathbf r(y,t)$ 
while \eqref{eq: Holder rest} converges to $0$ as $N\to \infty$. 
To deal with $m_{[Ny]}^{-1}\lgs{ p_{[Ny]} (Nt) }$, we proceed similarly, 
{using \eqref{eq: replacement} in} Lemma \ref{lem: replacement}, to get the analog of \eqref{eq: regularized integral}. 
\end{proof}
Finally, thanks to the bound \eqref{eq: bounded r and p} and the pointwise convergence result in Proposition \ref{prop: pointwise convergence}, 
and thanks to {using \eqref{eq: replacement} in} Lemma \ref{lem: replacement} for the field $\mathcal P_N$, 
we derive  (\ref{R limit}-\ref{P limit}) by applying the dominated convergence theorem.

%section
\section{Evolution of the energy $\mathcal E_N$}\label{sec: energy}

In this section we show the limit \eqref{E limit}. We will assume that $f \in \mathcal C^1([0,1])$. 
We can then recover the result \eqref{E limit} for $f \in \mathcal C^0([0,1])$ by density, and using the a priori estimate $\sum_{x} \lgs{e_x (t)} \le \Const N$ at all time $t \ge 0$. 

\subsection{Main decomposition of the energy}\label{sec:textbfm-decomp-energ}
In order to derive the limit of $\mathcal E_N$, we separate the contribution to the total energy from the temperature (that does not evolve with time) 
and from mechanical energy, i.e.\@ the average kinetic and potential energy (that does evolve due to the transport of momentum and displacement). 

At the macroscopic level, we deduce from (\ref{eq:equa diff r}-\ref{eq:equa diff e}) that 
\begin{align}
	\mathbf e(y,t) 
	&= \mathbf e (y,0) + \frac{1}{\overline m} \int_0^t \partial_y  (\mathbf r(y,s) \mathbf p(y,s)) \dd s \nonumber\\
	&= \frac1{\beta(y)} + \frac{\mathbf p{^2} (y,0)}{2 \overline m} + \frac{\mathbf r{^2} (y,0)}{2} + \frac{1}{\overline m} \int_0^t \partial_s \Big( \mathbf p^2(y,s)/2 + \overline m \mathbf r{^2}(y,s)/2 \Big) \dd s \nonumber\\
	&= \frac{\mathbf p^2 (y,t)}{2 \overline m} + \frac{\mathbf r^2(y,t)}{2} + \frac1{\beta(y)}.\label{eq: evolution energy limit}
\end{align}
At the microscopic level, we decompose
\begin{align}
	\mathcal E_N(f,t)
	=& \frac1N \sum_{x=1}^N f(x/N) \left(  \frac{\lgs{ p_x^2 } }{2m_x} + \frac{\lgs{ r_x^2 }}{2}  \right) (Nt) 
	\nonumber\\
	=& \frac1N \sum_{x=1}^N  f(x/N) \left(  \frac{\lgs{ p_x}^2}{2m_x} + \frac{\lgs{ r_x }^2}{2}  \right) (Nt)
	+\frac1N \sum_{x=1}^N  f(x/N) \left(  \frac{\lgs{ \widetilde p_x^2 }}{2m_x} + \frac{\lgs{ \widetilde r_x^2 }}{2}  \right) (Nt) 
	\nonumber\\ 
	=:& \mathcal A_N(t) + \mathcal F_N(t), 
	\label{eq: average and fluctuation}
\end{align}
with 
$$ 
	\widetilde p_x = p_x - \lgs{p_x}, \qquad \widetilde r_x = r_x - \lgs{r_x},
$$
and where $\mathcal A$ and $\mathcal F$ stands respectively for ``average'' and ``fluctuations''. 
Comparing \eqref{eq: evolution energy limit} and \eqref{eq: average and fluctuation}, we conclude that it is enough to show that, $\mathsf P$ almost surely, as $N\to \infty$, 
\begin{align} 
	& \mathcal A_N(t) \quad \to \quad  \int_0^1 f (y) \left(  \frac{\mathbf p^2 (y,t)}{2 \overline m} + \frac{\mathbf r^2(y,t)}{2} \right) \dd y, 
	\label{eq: 1st limit energy}\\
	& \mathcal F_N (t) - \mathcal F_N (0) \quad \to \quad  0, \qquad \mathcal F_N (0) \quad \to \quad \int_0^1 \frac{f(y)}{\beta (y)} \dd y. 
	\label{eq: 2d limit energy}
\end{align}

The limit \eqref{eq: 1st limit energy} is deduced in the same ways as (\ref{R limit}-\ref{P limit}): 
{
We first use \eqref{eq: replacement bis} in Lemma \ref{lem: replacement} to replace $\lgs{p_x}^2/2m_x$ by $\frac{\overline m}{2}(\lgs{p_x}/m_x)^2$
up to an error that vanishes as $N\to \infty$. }
Next, thanks to the bound \eqref{eq: bounded r and p} and the pointwise convergence result in Proposition \ref{prop: pointwise convergence}, we derive \eqref{eq: 1st limit energy}  by applying the dominated convergence theorem. 

{
The limit \eqref{eq: 2d limit energy} express the fact that the profile of thermal energy (temperature) remains frozen in time. }
It will be established thanks to the localization of the high modes of the chain; 
this is the only place where localization is used. 
Moreover, we will show in Section \ref{subsec: thermal equilibrium} that in thermal equilibrium, 
the equality $\mathcal F_N (t) = \mathcal F_N (0)$ holds without any assumption on the distribution of the masses (besides positivity).
This shows thus that Theorem \ref{the: main result} holds actually also for a clean chain if $\beta$ is constant.  

{
  \begin{Remark}\label{frozenforever}
    The proof of \eqref{eq: 2d limit energy} holds for any larger time scale $N^\alpha t$, $\alpha \ge 1$, i.e. 
    \begin{align} 
	& \mathcal F_N (N^{\alpha -1} t) - \mathcal F_N (0) \quad \to \quad  0.
	\label{eq: 2d limit energy-s}
\end{align}
  \end{Remark}
}

\subsection{Convergence of $\mathcal F_N (t)$}
To deal with the limit \eqref{eq: 2d limit energy}, we will use the fact that any mode of the chain at positive energy is spatially localized in the thermodynamic limit. 
Hence, we will expand the solutions to the equations of motion into the eigenmodes of the chain. 
In Section \ref{sec: eigenmodes expansion} below, we carry this expansion in details and we deduce the needed localization estimates. 
For our present purposes, it suffices to know the following: 
There exists a basis $\{\psi^k \}_{0\le k \le N-1}$ of $\R^N$, the basis of the eigenmodes of the chain, so that the solutions to the equations of motion read
\begin{align}
&\widetilde r_x (t)   = \sum_{k=1}^{N-1} \Big(  \frac{\langle  \nabla_+ \psi^k, \widetilde r(0) \rangle}{\omega_k} \cos \omega_k t + \langle \psi^k , \widetilde p(0) \rangle \sin \omega_k t \Big)  \frac{(\nabla_+ \psi^k)_x}{\omega_k}, 
	\label{eq:r solution bis}\\
	& \widetilde p_x(t) = \sum_{k = 0}^{N-1} \Big( \langle \psi^k, \widetilde p(0) \rangle \cos \omega_k t  -  \frac{\langle \nabla_+ \psi^k ,\widetilde r(0) \rangle}{\omega_k} \sin \omega_k t \Big) (M \psi^k)_x,
	\label{eq:p solution bis}
\end{align}
where $\omega_0 = 0$  and $\omega_k >0$ for $1 \le k \le N$ are the corresponding eigenfrequencies of the chain
{(we assume that the $\omega_k$ are sorted by increasing order)}.
Observe that the first term starts from $k=1$ while the second one starts from $k=0$.
Moreover, the orthogonality relation $\langle \psi^k , M \psi^j \rangle = \delta (k - j)$ holds and $\{ \omega_k^{-1}\nabla_+ \psi^k  \}_{1 \le k \le N-1}$ forms an orthonormal basis of $(\R^{N-1}, \langle \cdot, \cdot \rangle_{N-1})$. 
See Section  \ref{subsec: eigenmodes solution} for more details. 
This representation is useful to exploit localization: all modes with $k\gtrsim \sqrt N$ are spatially localized. 
See Section \ref{subsec: localization} for more quantitative estimates. 
However, low modes with $k\lesssim \sqrt N$ remain extended, and we will have to show that the contribution of these modes vanish since their proportion $N^{1/2}/N \to 0$ in the thermodynamic limit. 
Below, for technical reasons, we will replace $1/2$ by $1-\alpha$, for some $\alpha > 0 $ that we will need to choose small enough.  

Let $0 < \alpha \ll 1$, let 
$$ 
	F_1 = \Z \cap [0,N^{1-\alpha}], \qquad F_{2} = \Z \cap ]N^{1-\alpha}, N-1],
$$
and let us decompose $\widetilde r(t) = \widetilde r^{(1)}(t) + \widetilde r^{(2)}(t)$ and $\widetilde p (t) = \widetilde p^{(1)} (t) + \widetilde p^{(2)}(t)$ with 
$$
	\widetilde r^{(i)}(t) = \sum_{k\in F_{i} \backslash \{0\}} (\dots), \qquad 
	\widetilde p^{(i)}(t) = \sum_{k\in F_{i}} (\dots),
$$
for $i=1,2$ and with $(...)$ the summand featuring in \eqref{eq:r solution bis} or \eqref{eq:p solution bis}.
We insert this decomposition in $\mathcal F_N$: % (we omit the dependence on time in some expression for readibility): 
$$ 
	\mathcal F_N (t) =  \frac1N \sum_{x=1}^N  f(x/N) \left(  
	\frac{\blgs{ \big( \widetilde p_x^{(1)} + \widetilde p_x^{(2)} \big)^2 }}{2m_x} 
	+ \frac{\blgs{ \big(\widetilde r_x^{(1)} +  \widetilde r_x^{(2)} \big)^2 } }{2}  \right) (Nt). 
$$
Let us show the two following limits: 
\begin{align}
	&\mathcal F^{(1)}_N (t) = \frac1N \sum_{x=1}^N  f (x/N) \left(  
	\frac{\blgs{ \big( \widetilde p_x^{(1)} \big)^2} }{2m_x} 
	+ \frac{\blgs{ \big(\widetilde r_x^{(1)} \big)^2 } }{2}  \right) (Nt)  \quad \to \quad  0, 
	\label{eq:F1}\\
	&\mathcal F^{(2)}_N (t) = \frac1N \sum_{x=1}^N  f(x/N) \left(  
	\frac{\blgs{ \big( \widetilde p_x^{(2)}  \big)^2} }{2m_x} 
	+ \frac{\blgs{ \big(\widetilde r_x^{(2)}  \big)^2 } }{2}  \right) (Nt)  \quad \to \quad  \int\frac{f(y)}{\beta (y)} \dd y,
	\label{eq:F2}
\end{align}
which, by Cauchy-Schwarz inequality, implies \eqref{eq: 2d limit energy}. 

Let us show \eqref{eq:F1}.
Let us bound $|f (x/N)| \leq \Const$, and use the explicit solution (\ref{eq:r solution bis}-\ref{eq:p solution bis}): 
\begin{align*}
	&\quad|\mathcal F^{(1)}_N (t)| \\
	&\le \frac\Const{2N}\sum_{x=1}^N \frac{1}{m_x}
	\BBlgs{\bigg( 
	\sum_{k\in F_1} \Big(   \langle \psi^k ,\widetilde p (0) \rangle  \cos (\omega_k Nt) - \tfrac{\langle \nabla_+ \psi^k,\widetilde r(0)\rangle}{\omega_k} \sin (\omega_k Nt ) \Big) 
	m_x \psi^k_x \bigg)^2} \\
	&+\frac\Const{2N}\sum_{x=1}^N
	\BBlgs{\bigg( 
	\sum_{k\in F_1\backslash \{0\}} \Big( \tfrac{\langle \nabla_+ \psi^k, \widetilde r(0)\rangle}{\omega_k} \cos (\omega_k Nt) + \langle \psi^k, \widetilde p (0) \rangle \sin (\omega_k Nt) \Big) 
	\frac{(\nabla_+ \psi^k)_x}{\omega_k}
	\bigg)^2}.
\end{align*}
In both terms, one may expand the square so as to get a double sum over $k,j\in F_1$ or $k,j\in F_1 \backslash \{0\}$, and insert the sum over $x$ inside the sum over $k,j$. 
This yields
$$ 
	\sum_{x=1}^N \frac{m_x^2}{m_x} \psi^j_x \psi^k_x  = \langle \psi^j, M\psi^k \rangle = \delta(k-j), \qquad 
	\sum_{x=1}^N \frac{(\nabla_+ \psi^j)_x (\nabla_+ \psi^k)_x}{\omega_j \omega_k} = \delta (k-j).
$$
Thus 
\begin{align*}
	|\mathcal F^{(1)}_N (t)| \le
	&\frac\Const{2N} \sum_{k \in F_1} \Blgs{ \Big(  \langle \psi^k ,\widetilde p (0) \rangle  \cos (\omega_k Nt) - \frac{\langle \nabla_+ \psi^k ,\widetilde r(0)\rangle}{\omega_k} \sin (\omega_k Nt)  \Big)^2 } \\
	&+\frac\Const{2N} \sum_{k \in F_1 \backslash \{0\}} \Blgs{ \Big(  \frac{\langle \nabla_+ \psi^k ,\widetilde r(0)\rangle}{\omega_k} \cos (\omega_k Nt) + \langle \psi^k, \widetilde p(0) \rangle \sin (\omega_k Nt)  \Big)^2}.
\end{align*}
At this point, it suffices to show that there exists a constant $\Const$ such that, for all $k \in F_1$, 
\begin{equation}
\label{eq:avant30}
\lgs{ \langle \psi^k,\widetilde p(0) \rangle^2 } \le \Const, \qquad \frac{\lgs{ \langle \nabla_+ \psi^k, \widetilde r(0) \rangle^2 }}{\omega^2_k} \le \Const, 
\end{equation}
since, bounding $\sin$ and $\cos$ by 1, and using Cauchy-Schwarz, we obtain 
$$
	|\mathcal F^{(1)}_N (t)| \le \frac{\Const}{N} \sum_{k \in F_1} 1 = \frac{\Const N^{1 - \alpha}}{N} \to 0.
$$
Let us deal with $\lgs{ \langle \psi^k,\tilde p(0) \rangle^2 }$ (the other case is analogous): 
\begin{equation*}
\begin{split}
\lgs{ \langle \psi^k , \tilde p (0) \rangle^2 }
	&= \Blgs{\Big( \sum_x \psi^k_x  \widetilde p_x (0)  \Big)^2}
	= \Blgs{ \sum_{x,y}   \psi^k_x  \psi^k_y \widetilde p_x (0) \widetilde p_y (0)  }\\
	&= \sum_x (\psi^k_x)^2  \lgs{ (\widetilde p_x (0))^2 }
\end{split}
\end{equation*}
where we have used the fact that $\lgs{\cdot}$ is a product measure and that $\lgs{\widetilde p_x(0)} = 0$
 for all $x\in \{1, \dots , N\}$. 
We compute $ \lgs{ (\widetilde p_x (0))^2 } = \frac{m_x}{\beta(x/N)}$. 
Since $\beta$ is positive and continuous on $[0,1]$, there exists $\beta_->0$
 such that $\beta (x/N) \ge \beta_-$ for all $x\in \{1, \dots , N\}$. 
Hence
\begin{equation}\label{eq: beta minus equation}
	\lgs{ \langle \psi^k, \widetilde p (0) \rangle^2 } \le \frac{1}{\beta_-} \sum_{x=1}^N m_x (\psi^k_x)^2 = \frac{1}{\beta_-} \langle \psi^k , M \psi^k \rangle=\frac{1}{\beta_-}. 
\end{equation}
Let us now show \eqref{eq:F2}. 
A computation using the initial measure shows that 
$\mathcal F(0) \to \int \frac{f(y)}{\beta (y)} \dd y$ as $N\to \infty$. 
Hence, thanks to \eqref{eq:F1}, it holds that $\mathcal F^{(2)}_N (0)
 \to  \int \frac{f(y)}{\beta (y)} \dd y$ as $N\to \infty$. 
Thus it suffices to show that $\mathcal F^{(2)}_N (t) - \mathcal F^{(2)}_N (0) \to 0$ as $N\to \infty$. 
Let us write $\mathcal F^{(2)}_N(t)$ as a scalar product and expand it in the eigenmodes basis: 
\begin{align*}
	\mathcal F^{(2)}_N(t) 
	=&  \frac{1}{2N}\blgs{ \langle (f \cdot \widetilde p^{(2)})(Nt), M^{-1} \widetilde p^{(2)} (Nt)\rangle  + \langle ( f \cdot \widetilde r^{(2)} )(Nt), \widetilde r^{(2)}(Nt) \rangle  }\\
	=& \frac{1}{2N} \sum_{k\in F_2}\Blgs{ \langle ( f \cdot \widetilde p^{(2)})(Nt), \psi^k \rangle \langle \psi^k, \widetilde p (Nt)\rangle \\
	&\phantom{ \frac{1}{2N} \sum_{k\in F_2}\langle\langle} +\frac{1}{\omega_k^2}  \langle (f \cdot \widetilde r^{(2)}) (Nt), \nabla_+ \psi^k \rangle \langle \nabla_+\psi^k ,  \widetilde r(Nt) \rangle} .
\end{align*}
Here $g\cdot h$ denotes a function on $\Z\cap [1,N]$ obtained by the usual multiplication in real space between a function $g$ on $[0,1]$ and $h$ on $\Z\cap [1,N]$, i.e.\@ $(g\cdot h)_x = g(x/N) h_x$.
By Lemma \ref{lem: localization} stated in Section \ref{sec: eigenmodes expansion} below, one may associate a localization center $x_0(k)$ to each mode $\psi^k$ with $k\in F_2$: 
$x_0(k)$ is the center of the interval $J(k)$ featuring there 
(assuming that $\alpha$ is small enough so that the hypotheses of Lemma \ref{lem: localization} are satisfied).
For each $k\in F_2$, let us decompose $f$ as 
$$
	f = f_0 (k) + \widetilde f_k \quad \text{with} \quad f_0 (k) = f (x_0(k)/N) 
$$
(thus $f_0(k)$ is a constant and $\widetilde f_k$ vanishes at $x_0 (k)/N$).
We insert this decomposition in the above expression for $\mathcal F_N^{(2)}(t)$: 
\begin{align}
	\mathcal F^{(2)}_N(t) 
	= &\frac{1}{2N} \sum_{k\in F_2} f_0 (k) \,\BBlgs{ \langle \widetilde p(Nt) , \psi^k \rangle^2 + \frac{\langle \widetilde r (Nt),\nabla_+ \psi^k \rangle^2}{\omega_k^2}   } \label{eq: F2 part 1} \\
	&+ \frac1{2N}\sum_{k\in F_2} \BBlgs{ \langle (\widetilde f_k \cdot \widetilde p^{(2)})(Nt), \psi^k \rangle \langle \psi^k, \widetilde p (Nt)\rangle \nonumber\\
	&\phantom{\frac1{2N}\sum_{k\in F_2} \Blgs{.}} +\frac{1}{\omega_k^2}  \langle (\widetilde f_k \cdot \widetilde r^{(2)}) (Nt), \nabla_+ \psi^k \rangle \langle \nabla_+\psi^k ,  \widetilde r(Nt) \rangle}. \label{eq: F2 part 2}
\end{align}
Each expression between $\lgs{\dots}$ in the sum in \eqref{eq: F2 part 1} represents the energy of the mode $\psi^k$ and does not evolve with time, 
see \eqref{eq: energy of the modes} in Section \ref{sec: eigenmodes expansion} below. 
Hence, to show $\mathcal F^{(2)}_N (t) - {\mathcal F}^{(2)}_N (0) \to 0$, we only need to show that the sum in \eqref{eq: F2 part 2} converges to 0 as $N\to \infty$.

Let us consider a single term in the sum \eqref{eq: F2 part 2}, and let us focus on the term involving $\widetilde p$ (the one involving $\widetilde r$ is treated the same way). 
First, by Cauchy-Schwarz, 
\begin{equation}\label{eq: to show for localization}
\blgs{\langle (\widetilde f_k \cdot\widetilde p^{(2)})(Nt), \psi^k \rangle \langle \psi^k, \widetilde p (Nt)\rangle }
\le \lgs{\langle (\widetilde f_k \cdot \widetilde p^{(2)})(Nt), \psi^k \rangle^{2}}^{1/2} \lgs{ \langle \psi^k, \widetilde p (Nt)\rangle^{2}}^{1/2}.
\end{equation}
The second factor in \eqref{eq: to show for localization} is bounded by a constant: 
\begin{align}
	 \lgs{ \langle \psi^k, \widetilde p (Nt)\rangle^{2}}
	 &= \BBlgs{ \Big( \langle \psi^k , \widetilde p(0) \rangle \cos \omega_k Nt - \frac{\langle \nabla_+ \psi^k , \widetilde r(0) \rangle}{\omega_k} \sin \omega_k Nt  \Big)^2 }  \nonumber\\
	 &\le 2 \,\BBlgs{  \langle \psi^k , \widetilde p(0) \rangle^2 + \frac{\langle \nabla_+ \psi^k , \widetilde r(0) \rangle^2}{\omega_k^2} }\, \le \Const,\label{eq: to show for localization next 2}
\end{align}
see \eqref{eq:avant30}.
For the first factor in \eqref{eq: to show for localization}, we use again Cauchy-Schwarz to get
\begin{equation}
\label{eq: to show for localization next}
\begin{split}
	\lgs{\langle (\widetilde f_k \cdot \widetilde p^{(2)})(Nt), \psi^k \rangle^{2}}
	&= \Blgs{\Big( \sum_x \widetilde f_k(x/N) \widetilde p^{(2)}_x (Nt) \psi^k_x \Big)^2}\\
	&\le \Big( \sum_x \widetilde f_k^2 (x/N) (\psi^k_x)^2 \Big) \,\Blgs{\sum_x \big( \widetilde p^{(2)}_x\big)^2 (Nt)}\, .
	\end{split}
\end{equation}
For $\widetilde f_k$, we have the bound 
$$
	|\widetilde f_k (x/N)| = |f_k (x/N) - f_k (x_0(k)/N) |\le \Const \,\frac{|x-x_0 (k) |}{N}
$$
(this is the only place where we use $f \in \mathcal C^1([0,1])$).
Hence, form Lemma \ref{lem: localization} below, we deduce that for any $\epsilon > 0$, the first factor in the right hand side of \eqref{eq: to show for localization next} can be bounded by $1/N^{2-\epsilon}$ by taking $\alpha > 0$ small enough.
From the conservation of energy (see (\ref{eq: energy of the modes}) and the bound (\ref{eq:avant30})) , the second factor in \eqref{eq: to show for localization next} is $\mathcal O (N)$.
Hence, for $\alpha>0$ small enough, \eqref{eq: to show for localization next} goes to zero as $N\to \infty$. 
Combining this with \eqref{eq: to show for localization next 2}, we find that  \eqref{eq: to show for localization} goes to zero as $N \to \infty$, and hence that \eqref{eq: F2 part 2} converges to 0 as $N\to \infty$.
%hence the first factor in \eqref{eq: to show for localization} goes to zero as $N \to \infty$. 

\subsection{Thermal equilibrium case}\label{subsec: thermal equilibrium}

% If $\beta$ is constant in space we can easily that $\mathcal F_N (t) = \mathcal F_N (0)$, without conditions on the 
% probability distribution of the masses, i.e. valid also in the deterministic case (as discussed in the introduction).
Assume here that
 there exists some $\overline{\beta} >0$ such that $\beta (y) = \overline{\beta}$ for all $y \in [0,1]$. 
Then, we may relax the assumptions on the masses: requiring only that they are all strictly positive, let us show that $\mathcal F_N (t) = F_N (0)$ for all $t \ge 0$. 
This results from an exact computation. 

Since $f$ is arbitrary, it is necessary and sufficient to prove that, for any $x$, 
$$ 
	\frac{\dd }{ \dd t} \left(  \frac{\lgs{\tilde p_x^2(t)}}{2 m_x} + \frac{\lgs{\tilde r_x^2(t)}}{2} \right) = 0. 
$$
We compute 
\begin{multline*}
	\frac{\tilde p_x^2(t)}{2 m_x} 
	=\frac12 \sum_{j,k} \Big( \langle \psi^j , \tilde p(0) \rangle \cos \omega_j t - \frac{\langle \nabla_+ \psi^j , \tilde r(0)\rangle}{\omega_j} \sin \omega_j t  \Big)
	 {\,\times} \\
	 \Big( \langle \psi^k, \tilde p(0) \rangle \cos \omega_k t - \frac{\langle \nabla_+ \psi^k , \tilde r(0)\rangle}{\omega_j} \sin \omega_k t  \Big)  m_x \psi^j_x \psi^k_x.
\end{multline*}
A similar expression holds for $\tilde r^{2}(t)/2$. 
In order to obtain the expectation with respect to $\lgs{\cdot}$, we compute
\begin{eqnarray}
		\blgs{\langle \psi^k , \tilde p (0) \rangle \langle \psi^j , \tilde p (0) \rangle} & = & \frac{\delta (k-j)}{\overline{\beta}}, \label{cancellation beta const 1} \\
		\blgs{\langle \psi^k , \tilde p (0) \rangle \langle \nabla_+ \psi^j , \tilde r (0) \rangle} &  = & 0, \label{cancellation beta const 2} \\
		\blgs{\langle \nabla_+ \psi^k , \tilde r (0) \rangle \langle \nabla_+ \psi^j , \tilde r (0) \rangle} &  = & \frac{\omega_k^2 \delta (k-j)}{\overline\beta}. \label{cancellation beta const 3}
\end{eqnarray}
These three properties result from the fact the product structure of $\rho_{\mathrm{loc}}$, from the fact that the variables $\tilde p$ and $\tilde r$ are centered, 
and from the the fact that $\beta$ is constant for  \eqref{cancellation beta const 1} and \eqref{cancellation beta const 3}.
E.g.\@ \eqref{cancellation beta const 1} is obtained by 
\begin{equation*}
\begin{split}
\blgs{\langle \psi^k , \tilde p (0) \rangle \langle \psi^j , \tilde p (0) \rangle} 
	&= \sum_{x,y}\psi^k_x  \psi^j_y \lgs{ \tilde p_x (0) \tilde p_y (0)} = \frac{1}{\beta} \sum_x m_x \psi^k_x \psi^j_x \\
	&= \frac{\delta (k-j)}{\overline \beta}.
\end{split}
\end{equation*}
Hence we have that
$$
	 \frac{\lgs{\tilde p_x^2(t)}}{2 m_x} = \frac{1}{2\beta} \sum_k( \cos^2 \omega_kt + \sin^2 \omega_k t ) m_x (\psi^k_x)^2 = \frac{1}{2\beta}
$$
and similarly $\lgs{\tilde r_x^2(t)}/2 = \frac1{2\beta}$.

%section
\section{Eigenmodes expansion: integrability, localization}\label{sec: eigenmodes expansion}

We describe an explicit solution to the equations of motion \eqref{eq: equations of motion} in terms of the eigenmodes of the system. 
This representation is useful to establish the integrability of the system and to exploit the localization at all energies above the ground states (in the thermodynamic limit).  

\subsection{Solution to the equations of motion}\label{subsec: eigenmodes solution} 
From \eqref{eq: equations of motion}, one can deduce second order equations for $r$ and $p$ separately:
$$ 
	 \ddot r_x = \big(\nabla_+ M^{-1} \nabla_- r\big)_x \quad (1 \le x \le N-1), \qquad \ddot p_x = \big( \Delta M^{-1}p\big)_x \quad (1 \le x \le N),
$$ 
where, besides the boundary conditions $r_0 = r_N = 0$, we have assumed free boundary conditions for $M^{-1}p$, i.e. $m_0^{-1}p_0 = m_1^{-1}p_1$ and $m_{N+1}^{-1}p_{N+1} = m_N^{-1}p_{N}$. 
Notice that there are two different vector spaces: a $(N-1)$-dimensional space for $r$ with fixed boundary conditions, and a $N$-dimensional space for $M^{-1}p$ with free boundary conditions. % on $M^{-1}p$. 
Moreover, we observe that $\nabla_+ = - (\nabla_-)^\dagger$ with fixed boundary conditions, and that $\Delta = \Delta^\dagger$ with free boundary conditions. 

In order to solve the equations of motion, we need to diagonalize two matrices: 
$\big(\nabla_+ M^{-1} \nabla_-\big)^\dagger=\nabla_+ M^{-1} \nabla_-$ (of size $N-1$) 
and $\big(\Delta M^{-1} \big)^\dagger=M^{-1}\Delta$ (of size $N$). 
{Let us start with the latter:}
This matrix is not symmetric but the matrix $ M^{-1/2} (-\Delta) M^{-1/2} $ is symmetric and non-negative.
It admits thus an orthonormal set of eigenvectors, $\{\varphi^k\}_{0 \le k \le N-1}$ and corresponding eigenvalues $\omega^2_k$
{that we assume to be sorted by increasing order}. 
Moreover, the spectrum is $\mathsf P$-almost surely non-degenerate 
(see e.g.\@ Proposition II.1 in \cite{kunz_souillard}, considering here a perturbation around the non-degenerate equal masses case). 
Therefore the vectors $\psi^k = M^{-1/2} \varphi^k$ are such that 
\begin{equation}\label{eq: eigenmodes equation}
	M^{-1} (-\Delta) \psi^k = \omega^2_k  \psi^k, \qquad \langle \psi^j, M \psi^k \rangle = \delta(j-k).
\end{equation}
Because of free boundary conditions, $\omega^2_0 = 0$, and {one may chose $\psi_0$ to be given by}
$$
	\psi_0 = \Big(\sum_{x}m_x\Big)^{-1/2} (1, \dots , 1)^\dagger .   
$$
{Next}, the matrix ${-} \nabla_+ M^{-1} \nabla_- $ is symmetric and non-negative, 
and we denote the eigenvectors by $|\widetilde\psi^k \rangle$ for $1 \le k \le N-1$.
It is readily checked that {they may be chosen to be given by}
$$
	\widetilde \psi^k = \frac{1}{\omega_k} \nabla_+ \psi^k
$$
with the corresponding eigenvalue given by $\omega_k$ for $1 \le k \le N-1$. 
We observe that, by the free boundary conditions on $\psi^k$,  $\widetilde \psi^k (0) = \widetilde \psi^k (N) = 0$. 
 
Given initial conditions $r(0),p(0)$, we can write an explicit solution for $r(t),p(t)$: 
$$ 
	\langle \widetilde \psi^k , \ddot{r} \rangle = -\omega^2_k \langle \widetilde \psi^k, r\rangle \quad (1 \le k \le N-1), 
	\qquad 
	\langle \psi^k, \ddot p \rangle = -\omega^2_k \langle \psi^k, p \rangle \quad (0 \le k \le N-1).
$$
Thus
\begin{align*}
	&\langle \widetilde\psi^k, r(t) \rangle = \langle \widetilde \psi^k, r(0) \rangle \cos \omega_k t + \frac{\langle \widetilde\psi^k, \nabla_+ M^{-1}p(0) \rangle}{\omega_k} \sin \omega_k t \quad (1 \le k \le N- 1),\\
	&\langle \psi^k ,p(t) \rangle = \langle \psi^k ,p(0) \rangle \cos \omega_k t + \frac{\langle \nabla_- r(0),\psi^k \rangle}{\omega_k} \sin \omega_k t \quad (0 \le k \le N-1)
\end{align*}
with the convention $\frac{\sin 0}{0} = 1$ in the second expression at $k=0$ (notice that $\langle \nabla_- r(0) ,\psi^k \rangle = - \langle r(0) , \nabla_+\psi^k \rangle = 0$ for $k=0$). This yields therefore
\begin{align}
	&r(t) = \sum_{k=1}^{N-1} \Big(  \frac{\langle  \nabla_+ \psi^k ,r(0) \rangle}{\omega_k} \cos \omega_k t + \langle \psi^k , p(0) \rangle \sin \omega_k t \Big)  \frac{\nabla_+ \psi^k}{\omega_k}, 
	\label{eq:r solution}\\
	&p(t) = \sum_{k = 0}^{N-1} \Big( \langle \psi^k ,p(0) \rangle \cos \omega_k t  -  \frac{\langle \nabla_+ \psi^k ,r(0) \rangle}{\omega_k} \sin \omega_k t \Big) M \psi^k.
	\label{eq:p solution}
\end{align}

\subsection{Full set of invariant quantities}\label{subsec: invariant quantities} 
We observe that the dynamics has $N$ invariant quantities, corresponding to the energy of each mode. 
It is thus an integrable system. 
Indeed, let us write the full energy as 
\begin{align*}
	H &= \frac12 \big( \langle p, M^{-1}p \rangle + \langle r,r \rangle \big)
	= \frac12 \sum_{k=0}^{N-1} \langle p,\psi^k \rangle \langle M\psi^k, M^{-1}p \rangle + \frac12 \sum_{k=1}^{N-1} \langle r, \widetilde \psi^k \rangle \langle \widetilde \psi^k, r \rangle \\
	&= \frac12 \sum_{k=0}^{N-1} \Big(  \langle p,\psi^k \rangle^2 + \frac{\langle r ,\nabla_+ \psi^k \rangle^2}{\omega_k^2} \Big)
\end{align*}
with the convention that the second term in the last sum is $0$ at $k=0$.
From the evolution equation of the dynamics, one gets that actually 
\begin{equation}
\label{eq: energy of the modes}
	\frac{\dd }{ \dd t} \left(  \langle p ,\psi^k \rangle^2 + \frac{\langle r, \nabla_+ \psi^k \rangle^2}{\omega_k^2} \right) = 0 \quad\text{for all} \quad 0 \le k \le N-1. 
\end{equation}

Moreover, by taking specific linear combinations of these conserved quantities, one can obtain conserved quantities that can be written as a sum of local terms. 
This is for instance the case of the quantity $I$ defined in \eqref{eq: I conserved}, that reads also 
$$
	I(r,p) = \frac12 \sum_{k=1}^{N-1} \omega_k^2 \, \left(  \langle p,\psi^k \rangle^2 + \frac{\langle r,\nabla_+ \psi^k \rangle^2}{\omega_k^2} \right).
$$

\subsection{Localization}\label{subsec: localization} 
Localization can be expressed mathematically in the following strong sense, 
see \cite{kunz_souillard,aizenman_graff,aizenman_molchanov} for the general theory and \cite{verheggen,ajanki_huveneers} for precise estimates on the localization length as one approaches the ground state. 
Let $0 < \alpha < \frac12$ and let $I(\alpha) = ]N^{(1- \alpha)},N-1]\cap \Z$. 
There exist constants $\Const, \const >0$ such that 
$$
	\mathsf E \Big( \sum_{k \in I(\alpha)} |\psi^k_x \psi^k_y | \Big) \le \Const \ed^{-c|x-y|/\zeta(\alpha)} \quad \text{with} \quad \zeta(\gamma) = N^{2\alpha}.
$$
We will use this estimate to show that every mode in $k \in I(\alpha)$ is supported in a small subset of $[1,N] \cap \Z$ up to a small error: 

\begin{Lemma}\label{lem: localization}
Let $\alpha, \gamma>0$ be such that $2 \alpha < \gamma < 1$.
There exists almost surely $N_0 \in\N$ so that for all $N \ge N_0$, and for all $k \in I (\alpha)$, 
there exists an interval $J(k)$ with $|J(k)| \le 2N^\gamma$ such that $|\psi^k_x | \le N^{-1/\gamma}$ for all $x \notin J(k)\cap \Z$. 
\end{Lemma}

\begin{proof}
Let us first show that 
\begin{align}
	P &:= \mathsf P\big( \exists k \in I(\alpha), \exists x,y\in [1,N]\cap \Z : |x-y| \ge N^\gamma, |\psi^k_x | \ge N^{-1/\gamma},  |\psi^k_y | \ge N^{-1/\gamma}   \big)  
	\nonumber\\
	&\le\; \Const(\alpha, \gamma) \ed^{- N^{(\gamma - 2\alpha)/2}}
	\label{eq: proba 0 event localization}
\end{align}
Indeed we compute
\begin{align*}
	P &\le \sum_{k\in I(\alpha)} \sum_{x,y : |x-y| \ge N^\gamma} \mathsf P \big(  |\psi^k_x | \ge N^{-1/\gamma},  |\psi^k_y | \ge N^{-1/\gamma} \big) \\
	&\le \sum_{k\in I(\alpha)}  \sum_{x,y : |x-y| \ge N^\gamma} \mathsf P \big(  |\psi^k_x\psi^k_y | \ge N^{-2/\gamma} \big) \\
	&\le \sum_{k\in I(\alpha)}  \sum_{x,y : |x-y| \ge N^\gamma} N^{2/\gamma} \mathsf E ( |\psi^k_x\psi^k_y |) \\
	&\le \Const N^{2/\gamma} \sum_{x,y : |x-y| \ge N^\gamma} \ed^{-c|x-y|/\zeta(\alpha)} 
	\le \Const N^{2/\gamma} N^{2} \ed^{- N^{\gamma - 2 \alpha}} 
	\le  \Const(\alpha, \gamma) \ed^{- N^{(\gamma - 2\alpha)/2}}.
	%&\le C \, N^{2/\gamma + 3} \ed^{-N^{1/\gamma}/\zeta} \le C' \frac{1}{N^{2/\gamma}}
\end{align*}
Therefore, there exists almost surely $N_0$ so that for for all $N\ge N_0$, the event featuring in \eqref{eq: proba 0 event localization} does not occur. 
Hence in this case, for all $k \in I(\alpha)$ and all $|x - y|>N^{\gamma}$, we must have either $|\psi^k_x | \le 1/N^{1/\gamma}$ or $|\psi^k_y | \le 1/N^{1/\gamma}$.
This means thus that for any $k\in I(\alpha)$  there exists an interval $J(k)$ with $|J(k)| \le 2N^\gamma$ such that $|\psi^k_x | \le N^{-1/\gamma}$ for all $x \notin J(k)\cap \Z$. 
\end{proof}

\end{document}